\documentclass[english,numbers=noenddot,
  final,
  twoside=false,%
  ]{scrartcl}
\usepackage[utf8]{inputenc}
\usepackage[T1]{fontenc}
\usepackage[english]{babel}
\usepackage{amsmath, amsthm, amssymb}
\usepackage[final]{graphics}
\usepackage{uniinput}                           %
\usepackage[parfill]{parskip}                   %
\def\gettexliveversion#1(#2 #3 #4#5#6#7#8)#9\relax{#4#5#6#7}
\edef\texliveversion{\expandafter\gettexliveversion\pdftexbanner\relax}
\ifnum\texliveversion<2018
  \begingroup
    \makeatletter
    \@for\theoremstyle:=definition,remark,plain\do{%
      \expandafter\g@addto@macro\csname th@\theoremstyle\endcsname{%
        \addtolength\thm@preskip\parskip
      }%
    }%
  \endgroup
\else
  \RedeclareSectionCommand[afterskip=1.3\baselineskip]{section}
  \RedeclareSectionCommand[afterskip=1.025\baselineskip]{subsection}
  \RedeclareSectionCommand[afterskip=1.025\baselineskip]{subsubsection}
\fi

\usepackage[section]{algorithm}                  %
\usepackage[noend]{algpseudocode}
\usepackage{bbm}

\usepackage{enumitem}                           %
\usepackage{mathtools}                          %
\usepackage{float}                              %
\usepackage{dpfloat}                            %
\usepackage{longtable}
\usepackage{xparse}                             %
\usepackage{tikz}                               %
\usepackage{tikz-cd}
\usepackage{pgfplots}
\pgfplotsset{compat=1.14}
\usepackage{makecell}

\usepackage{silence}
\usepackage{thmtools}                           %
\usepackage{xcolor}                             %
\usepackage[hyperfootnotes=false,hypertexnames=false,final]{hyperref}  %
\usepackage[all]{hypcap}                        %
\usepackage[capitalise,noabbrev]{cleveref}

\expandafter\def\csname ver@etex.sty\endcsname{3000/12/31}  %
\usepackage{csquotes}                           %
\let\etoolboxforlistloop\forlistloop %
\usepackage{autonum}                            %
\let\forlistloop\etoolboxforlistloop %
\makeatletter
\let\blx@noerroretextools\@empty
\makeatother
\usepackage[
  backend=biber,
  style=alphabetic,
  useprefix=true,%
  maxalphanames=4,
  bibencoding=utf8,
  giveninits=true,
  maxbibnames=9,
  date=year,
  doi=true,isbn=false,url=true,eprint=true]{biblatex}
\DeclareSourcemap{
  \maps[datatype=bibtex]{
    \map[overwrite]{
      \step[fieldsource=doi, final]
      \step[fieldset=url, null]
      \step[fieldset=eprint, null]
    }
  }
}

\AfterEndPreamble{%
  \let\reforig\ref
  \let\ref\cref
  \newcommand{\myeqref}[1]{(\reforig{#1})} %
  \undef{\autoref}%
  \undef{\namecref}%
}%

\usepackage{ifdraft}
\ifdraft{%
  \definecolor{todo-color}{RGB}{23,130,132}
  
}{}

\makeatletter %
\def\thmt@refnamewithcomma #1#2#3,#4,#5\@nil{%
  \@xa\def\csname\thmt@envname #1utorefname\endcsname{#3}%
  \ifcsname #2refname\endcsname
    \csname #2refname\expandafter\endcsname\expandafter{\thmt@envname}{#3}{#4}%
  \fi
}
\makeatother

\declaretheorem[style=definition,numberwithin=section]{definition}  %
\declaretheorem[style=definition,qed=$\Diamond$,sibling=definition,name=Example,refname={Example,Examples}]{example2}
\newenvironment{example}{\par\begin{example2}\renewcommand{\qedsymbol}{\ensuremath\Diamond}}{\end{example2}}  %

\declaretheorem[style=plain,sibling=definition]{corollary}  %
\declaretheorem[style=plain,sibling=definition]{proposition}  %
\declaretheorem[style=plain,sibling=definition]{theorem}
\declaretheorem[style=plain,sibling=definition]{lemma}
\declaretheorem[style=definition,qed=$\Diamond$,sibling=definition,name=Remark,refname={Remark,Remarks}]{remark2}
\newenvironment{remark}{\par\begin{remark2}\renewcommand{\qedsymbol}{\ensuremath\Diamond}}{\end{remark2}}  %

\declaretheorem[style=remark,numbered=no,qed=$\qedsymbol$,name=Proof]{proof2}
\renewenvironment{proof}{\par\begin{proof2}}{\end{proof2}}

\usepackage{chngcntr}
\counterwithin{theorem}{section}

\makeatletter%
  \newenvironment{thm-enumerate}%
  {\enumerate[%
    ref=\protect\myown@ref{\thetheorem\,}{\labelenumi}]}%
  {\endenumerate}%
  \let\myown@ref\relax%
    \def\localref#1{{\let\myown@ref\@secondoftwo\labelcref{#1}}}%
\makeatother%
\makeatletter%
  \newenvironment{thm-enumerateB}%
  {\enumerate[%
    ref=\protect\myown@ref{\thetheorem\,}{\labelenumiiFull}]}%
  {\endenumerate}%
  \let\myown@ref\relax%
    \def\localref#1{{\let\myown@ref\@secondoftwo\labelcref{#1}}}%
\makeatother%

\input{widebar} %

\makeatletter
\g@addto@macro\bfseries{\boldmath}
\makeatother

\makeatletter
\renewcommand{\@pnumwidth}{1.75em}%
\renewcommand{\@tocrmarg}{2.75em}%
\makeatother

\newcommand{\spat}{{\kern 0.05em}}                      %
\let\theenumiOld\theenumi
\let\theenumiiOld\theenumii
\renewcommand{\labelenumi}{\textnormal{(\theenumi )}}   %
\newcommand{\labelenumiiFull}{\textnormal{(\theenumiOld\theenumiiOld )}} %

\Crefname{ALC@unique}{Step}{Steps}
\algrenewcommand\algorithmicrequire{\textbf{Input:}}
\algrenewcommand\algorithmicensure{\textbf{Output:}}
\algnewcommand\algorithmicassumptions{\textbf{Assumptions:}}
\algnewcommand\Assumptions{\item[\algorithmicassumptions]}

\newcommand{\themytitle}{Truncated moment problems on positive-dimensional algebraic varieties}
\newcommand{\themyauthor}{Markus Wageringel}
\title{\themytitle}
\author{\themyauthor{}\footnote{Osnabrück University, Institute of Mathematics, \texttt{markus.wageringel@uos.de}}}
\date{}
\hypersetup{
    pdftitle = \themytitle, pdfauthor = \themyauthor,
    colorlinks,
    linkcolor={blue!80!black},
    citecolor={green!50!black},
    urlcolor={blue!80!black}
}
\DeclarePairedDelimiter\prn{\lparen}{\rparen}
\newcommand{\lr}[1]{\prn*{#1}}

\NewDocumentCommand\mylist{O{:} >{\SplitList{,}}m}
{%
  \def\itemdelim{\def\itemdelim{#1}}%
  \ProcessList{#2}{\mylistitem}%
}
\newcommand\mylistitem[1]{\itemdelim #1}

\NewDocumentCommand\tsubstack{O{,\spat} >{\SplitList{\\}}m}
{%
  \def\itemdelim{\def\itemdelim{#1}}%
  \ProcessList{#2}{\tsubstackitem}%
}
\newcommand\tsubstackitem[1]{\itemdelim #1}

\newcommand{\coloneqq}{\mathrel{{\mathop:}{\mkern-0.5mu}{=}}}%

\DeclareMathOperator{\image}{im}
\DeclareMathOperator{\kernel}{ker}
\DeclareMathOperator{\Hom}{Hom}

\newcommand{\dirac}[1]{\delta_{#1}}

\newcommand{\dual}{^*}
\newcommand{\withoutzero}{^*}

\newcommand{\linterval}[1]{\left[#1\right)}
\renewcommand{\d}{\mathrm{d}}
\newcommand{\I}{\mathrm{i}}

\newcommand{\homog}[1]{%
  \let\olddots\dots{}%
  \let\oldldots\ldots{}%
  \let\dots\cdots{}%
  \let\ldots\cdots{}%
  \bracketed*{\mylist[:]{#1}} %
  \let\dots\olddots{}%
  \let\ldots\oldldots{}}

\newcommand{\compose}{\mathbin{\circ}}

\newcommand{\restrict}[1]{|_{#1}}

\newcommand{\K}{\mathbbm{k}}

\newcommand{\contin}[1]{C^{#1}}
\newcommand{\contincompact}[1]{\contin{#1}_{\textnormal{c}}}

\DeclarePairedDelimiter\bracketed{\lbrack}{\rbrack}
\DeclarePairedDelimiter\angled{\langle}{\rangle}
\DeclarePairedDelimiter\braced{\lbrace}{\rbrace}
\DeclarePairedDelimiter\prnabs{|}{|}

\newcommand{\T}{\mathbb{T}}  %

\def\mf{\mathfrak}

\newcommand{\totaldeg}[1]{\prnabs*{#1}}

\newcommand{\maxdeg}[1]{\totaldeg{#1}_{\infty}}
\newcommand{\idealspan}[2][]{\angled*{#2}_{#1}}

\newcommand{\transp}{^{\top}}  %

\newcommand{\conj}[1]{\overline{#1}}
\newcommand{\scalarprod}[3][]{\angled*{#2,#3}\ifstrempty{#1}{}{_{#1}}}
\newcommand{\abs}[1]{\prnabs*{#1}}

\DeclareMathOperator{\diag}{diag}%

\newcommand{\maxideal}[1]{\mathfrak{m}_{#1}}

\DeclareMathOperator{\ev}{ev}
\newcommand{\evd}[2]{\proj{#1, ≤#2}}
\newcommand{\kdual}[1]{\Hom_\K\lr{#1,\K}}
\newcommand{\semikdual}[1]{\Hom_\K^{\textnormal{semi}}\lr{#1,\K}}
\newcommand{\Transp}[1]{#1\transp}
\DeclareMathOperator{\VV}{V}
\newcommand{\V}[1]{\VV\lr{#1}}
\DeclareMathOperator{\nonVV}{D}
\newcommand{\nonV}[1]{\nonVV\lr{#1}}%
\newcommand{\submodulequotient}[2]{#1/(#2\cap#1)}
\makeatletter
\newcommand*\bigcdot{\mathpalette\bigcdot@{.5}}
\newcommand*\bigcdot@[2]{\mathbin{\vcenter{\hbox{\scalebox{#2}{$\m@th#1\bullet$}}}}}
\makeatother

\DeclareUnicodeCharacter{2297}{\ensuremath{\otimes}}  %

\newcommand{\ie}{i.\,e.\ }
\newcommand{\eg}{e.\,g.\ }

\DeclareMathOperator{\rk}{rk}  %

\DeclareMathOperator{\vanishingIdeal}{I}
\newcommand{\Id}[1]{\vanishingIdeal\lr{#1}}
\newcommand{\Z}[1]{\V{#1}}%
\newcommand{\resid}[1]{\widebar{#1}}

\newcommand{\closure}[1]{\widebar{#1}}
\newcommand{\zariski}[1]{\closure{#1}} %
\DeclareMathOperator{\supp}{supp}

\newcommand{\LambdaFunction}[2][0]{%
  \let\tmp\relax
  \newcommand\tmp[#1]{#2}\tmp}
\newcommand{\range}[4][,]{\LambdaFunction[1]{#2}{#3}#1%
  \ifthenelse{\equal{#1}{,}}{…}{\cdots}%
  #1\LambdaFunction[1]{#2}{#4}}

\newcommand{\proj}[1]{π_{#1}}
\newcommand{\sform}[2][]{\angled*{#2}_{#1}}%
\newcommand{\blank}{\mathord{-}}%
\newcommand{\invol}[2][]{#2^{\circ_{#1}}}
\newcommand{\Invol}[2][]{\lr{#2}^{\circ_{#1}}}

\newcommand{\fourier}[1]{\hat{#1}}

\newcommand{\weakstar}{weak${}\dual{\kern -0.05em}$}

\usepackage[shortcuts]{extdash}%

\newcommand\blfootnote[1]{%
  \begingroup%
  \renewcommand\thefootnote{}\footnote{#1}%
  \addtocounter{footnote}{-1}%
  \endgroup%
}
\begin{document}
\maketitle{}

\begin{abstract}
\small
\noindent
This manuscript transfers the main aspects of Prony's method
from finitely\-/supported measures to the classes of
signed or non-negative measures supported on algebraic varieties of any dimension.
In particular, we show that the Zariski closure of the support of these measures
is determined by finitely many moments
and can be computed from the kernel of certain moment matrices.
\blfootnote{\textup{2010} \textit{Mathematics Subject Classification}: 13P25, 62F99, 65T40}%
\end{abstract}

\addsec{Introduction}%

The truncated moment problem for finitely-supported measures
asks for parameter recovery from a given finite set of moments.
This problem can be addressed by a multivariate form of Prony's method \cite{kunis2016:pronymultiv,vdohe2017,sauer2017:pronymultiv,mourrain17:polyexp},
a widely-used tool in signal processing that is algebraic at heart.
It recovers the finitely many support points of such a measure
as the zero set of a family of polynomials,
so it is natural to view the support as a zero-dimensional algebraic variety.

In this manuscript,
we switch from finitely-supported measures
to the much more general class of
measures that are supported on algebraic varieties of any dimension
and we analyze which features of Prony's method can be transferred to this setting.
By considering the kernels of certain moment matrices,
we show that
it is possible to recover the vanishing ideal of the support of a measure.
In other words, given sufficiently many moments,
one obtains the Zariski closure of the support, by algebraic means.

\minisec{Contributions}

    The Vandermonde decomposition of the moment matrix
    of a finitely-supported signed measure
    is an essential ingredient of Prony's method.
    \ref{thm:hankelopfactorization}
    forms an analog of this decomposition
    that is suitable also for measures
    supported on positive-dimensional varieties.

    For any compactly-supported signed measure,
    \ref{thm:idealequalkernelextended} establishes a relationship
    between moment matrices and the vanishing ideal of the support.
    It shows that the Zariski closure of the support
    can be computed from finitely many moments.
    The \lcnamecref{thm:idealequalkernelextended}
    can be viewed as an extension of \ref{thm:idealequalkernel},
    which makes a similar statement for non-negative measures
    and has been considered by
    \cite{laurentrostalski2012,lasserre2021:empiricalmomentschristoffel},
    in the real affine (non-trigonometric) setting,
    and is also related to
    \cite{ongie15:piecewisesmooth,ongie2016:piecewiseconstant,vetterli2016}
    which have investigated the case of plane curves.

    Additionally, \ref{cor:idealequalkernelpolynomial} gives an extension
    for signed measures that are a product of a polynomial and a non-negative measure.
    This allows us to formulate a generalization of Prony's method
    in \ref{rem:pronyposdim}
    for particular measures that are not necessarily finitely-supported,
    but are supported on an algebraic variety of any dimension.
    Moreover, a variant for complex linear combinations of non-negative measures
    is proved in \ref{thm:idealequalkernel:mixture}.

\minisec{Outline}

After briefly summarizing the main ideas of the multivariate Prony method in \ref{sec:prony},
we start in \ref{sec:sesquilinearity} with a short treatment
of sesquilinear forms that can be associated to a functional $σ$.
The concept of sesquilinearity is useful in this context
as it allows us to treat both cases, that of measures in affine space
and on the complex torus, simultaneously.
We then continue in \ref{sec:factorization}
by transferring to the more general setting
the Vandermonde factorization of \ref{lem:vandermondefactorization:1}
that is such an essential ingredient for Prony's method.
\ref{sec:support} addresses our leading question,
that of recovering the algebraic variety
the measure $μ$ is supported on.
This can be achieved by using finitely many moments,
both for non-negative as well as compactly-supported signed measures,
in affine space and on the complex torus.
In case of non-negative measures,
the moment matrices are positive-semidefinite
which allows for stronger statements;
we illustrate this difference in some examples.

\minisec{Terminology}

The symbol~$\K$ always denotes a field;
in some sections we explicitly assume that it is of characteristic~$0$.
The (algebraic) dual space of a $\K$-vector space $V$
is written as $\kdual{V}$.
Similarly, $\semikdual{V}$ denotes the set of semilinear maps from $V$ to $\K$
(cf.~\ref{sec:sesquilinearity}).

For an introduction to algebraic geometry,
see \cite{Cox:2015}.
By \emph{algebraic variety}, we refer to
the vanishing set of a set of polynomials, also known as algebraic set,
that is, we do not require irreducibility.
A variety generated by an ideal $\mf{a}$ is denoted by $\V{\mf{a}}$.
The vanishing ideal of a set $X\subseteq\K^n$ is denoted by $\Id{X}$.
We use multi-index notation for monomials.
Thus, when working in the polynomial ring $\K[x_1,…,x_n]$,
the monomials are denoted by $x^α = x_1^{α_1}\cdots x_n^{α_n}$, $α∈ℕ^n$.
The \emph{(total) degree} of a polynomial $p = \sum_{α∈ℕ^n} p_α x^α$ with coefficients $p_α∈\K$
is given by $\deg(p) = \max\{\totaldeg{α}\mid α∈ℕ^n,\,p_α≠0\}$,
where $\totaldeg{α} \coloneqq α_1 + \cdots + α_n$.
Similarly, we define the \emph{max-degree} of a Laurent polynomial $q = \sum_{α∈ℤ^n} q_α x^α$, $q_α∈\K$,
as $\max\{\maxdeg{α}\mid α∈ℤ^n,\,q_α≠0\}$,
where $\maxdeg{α} \coloneqq \max\{\abs{α_1},…,\abs{α_n}\}$.
The same definition applies when $q$ is a polynomial.
Though, note that the max-degree does not define a grading of the polynomial ring,
but gives rise to a filtration (cf.~\ref{ex:filtrations}).

Given an ideal $\mf{a} \subseteq \K[x_1,…,x_n]$,
the Krull-dimension of the quotient ring $\K[x_1,…,x_n] / \mf{a}$,
\ie the supremum of the heights of all prime ideals,
is the same as the dimension of the variety $\V{\mf{a}} \subseteq \K^n$
(cf.~\cite[Theorem~9.3.8]{Cox:2015}).
By abuse of language, we also refer to this as the dimension of the ideal $\mf{a}$.
The residue class of a polynomial $p ∈ \K[x_1,…,x_n]$ modulo an ideal $\mf{a}$
is denoted by $\resid{p} = p + \mf{a}$.
By $\idealspan{\blank}$,
we denote the ideal spanned by a family of ring elements.
We write
$\maxideal{ξ} \coloneqq \idealspan{x - ξ} = \idealspan{x_1-ξ_1,…,x_n-ξ_n}$
for the maximal ideal associated to a point $ξ∈\K^n$.
Furthermore, the map
$\ev_{ξ}\colon \K[x_1,…,x_n]\to \K$, $p \mapsto p(ξ)$,
denotes the \emph{evaluation homomorphism} associated to a point $ξ∈\K^n$.
It can naturally be viewed as a ring homomorphism to
the quotient ring corresponding to the ideal $\maxideal{ξ}$.

Unless otherwise noted, the term \emph{measure}
refers to non-negative Borel measures.
Occasionally, we also work with signed measures.
Over the complex numbers,
the term \emph{signed measure}
stands for complex(-signed) measure.
Every (finite) non-negative measure is, in particular, a signed measure.
For details, we refer to \cite{schwartz1973,rudin1987}.

\section{Prony's method}\label{sec:prony}\label{sec:multivarprony}

The following is a multivariate generalization of \emph{Prony's method}
that, in its univariate form, goes back to \cite{prony1795}.
We wish to transfer its essence to the more general setting of algebraic varieties of any dimension.
The variant we cite here is useful for this,
but there are many alternative formulations that accentuate different points of view.
For instance, it has been considered in terms of exponential sums with a focus on signal processing
in \cite{kunis2016:pronymultiv,vdohe2017,sauer2017:pronymultiv,mourrain17:polyexp}.
Another variation of Prony's method
is \emph{Sylvester's algorithm} \cite{sylvester1886}.
It is also related to \emph{Macaulay inverse systems} (see \eg{}\cite[Chapter~21.2]{eisenbud})
and \emph{apolarity theory}
(cf.~\cite[Lemma~1.15, algorithm in Chapter~5.4]{iarrobinokanev99},
\cite[Chapter~19]{schmuedgen2017}),
which put more emphasis on algebraic and geometric aspects.

\begin{proposition}[{\cite{prony1795}, \cite{kunis2016:pronymultiv}, \cite[Remark~2.8, Corollary~2.19]{vdohe2017}}]
  \label{thm:pronyohe:0}
  \label{lem:vandermondefactorization}\crefalias{enumi}{proposition}\crefalias{enumii}{proposition}
  Let $\K$ be a field and let $R = \K[x_1,…,x_n]$ be the polynomial ring in $n$ variables.
  Let $σ = \sum_{j=1}^r λ_j \ev_{ξ_j}$ for $λ_j ∈ \K$
  and $ξ_j∈\K^n$, $1≤j≤r$.
  Let $d,d' ∈ ℕ$ and define $H_{d',d} \coloneqq \lr{σ(x^{α+β})}_{\totaldeg{α}≤d', \totaldeg{β}≤d}$.
  Then the following properties hold:
  \begin{thm-enumerate}
    \item\label{lem:vandermondefactorization:1}
      $H_{d',d} = V_{≤d'}\transp Λ V_{≤d}$,
      where $Λ \coloneqq \diag\lr{λ_1,…,λ_r}$ and
      $V_{≤d} \coloneqq \lr{ξ_j^α}_{1≤j≤r,\totaldeg{α}≤d}$.
    \item\label{lem:vandermondefactorization:2ab}
      If $λ_1,…,λ_r≠0$ and $\ev_{≤d'}\colon R_{≤d'} \to \K^r$, $x^α\mapsto (ξ_j^α)_{1≤j≤r}$, is surjective,
      then:
      \begin{thm-enumerateB}
        \item\label{lem:vandermondefactorization:2}
          $\kernel H_{d',d} = \kernel V_{≤d} = \Id{\{ξ_1,…,ξ_r\}} \cap R_{≤d}$.
        \item\label{thm:pronyohe}
          $\Z{\kernel H_{d',d}} = \{ξ_1,…,ξ_r\}$ if $d-1≥d'$.
      \end{thm-enumerateB}
  \end{thm-enumerate}
\end{proposition}
\begin{proof}
  The factorization $H_{d',d} = V_{≤d'}\transp Λ V_{≤d}$ follows by direct computation.
  Furthermore, if $λ_1,…,λ_r≠0$
  and $\ev_{≤d'}$ is surjective,
  then $V_{≤d'}\transp Λ$ represents an injective map,
  so the kernels of $V_{≤d}$ and $H_{d',d}$ are the same
  and agree with the truncated vanishing ideal $\Id{\{ξ_1,…,ξ_r\}} \cap R_{≤d}$,
  which shows \localref{lem:vandermondefactorization:2}.
  Then part~\localref{thm:pronyohe} follows from the observation
  that the surjectivity of $\ev_{≤d-1}$ implies
  $\V{\kernel V_{≤d}} = \{ξ_1,…,ξ_r\}$;
  see \cite[Theorem~2.15]{vdohe2017}.
\end{proof}
\par
Note that, if the points $ξ_1,…,ξ_r$ are not distinct,
then the map $\ev_{≤d'}\colon R_{≤d'} \to \K^r$
in \localref{lem:vandermondefactorization:2} can never be surjective,
so the surjectivity assumption implies in particular that the points are distinct.
Further, note that the matrix $H_{d',d}$ in \ref{thm:pronyohe:0}
represents the $\K$-linear map
into the dual space of the vector space $R_{≤d'}$ given by
\[
  R_{≤d} \longrightarrow \kdual{R_{≤d'}},\qquad p \longmapsto (q \mapsto σ(p q)),
\]
as well as the $\K$-bilinear mapping
\[
  R_{≤d'} × R_{≤d} \longrightarrow \K,\qquad (q, p) \longmapsto σ(p q).
\]

A map of the form $σ = \sum_{j=1}^r λ_j \ev_{ξ_j}$,
where $\ev_{ξ_j}$ denotes the evaluation homomorphism associated to the point $ξ_j$,
can also be viewed as \emph{exponential sum}.
It satisfies $σ(x^α) = \sum_{j=1}^r λ_j ξ_j^α$ for all $α∈ℕ^n$,
so can be interpreted as a map $ℕ^n \to \K$,
by composing it with $α \mapsto x^α$.

Also note that $σ$ is the \emph{moment functional} of the finitely-supported measure
$μ\coloneqq \sum_{j=1}^r λ_j \dirac{ξ_j}$,
where $\dirac{ξ_j}$ denotes the Dirac measure supported at the point $ξ_j ∈ \K^n$ for $1≤j≤r$.
For this interpretation, we usually assume that $\K$ is $ℝ$ or $ℂ$.
If $\K=ℂ$ and the weights $λ_1,…,λ_r∈ℂ$ are complex,
then $μ$ is a \emph{signed} (complex) measure,
which is explicitly allowed in this setting.
The signed measure $μ$ satisfies
$\int_{\K^n} x^α \d μ(x) = \sum_{j=1}^r λ_j ξ_j^α = σ(x^α)$,
so $σ(x^α)$ agrees with the \emph{$α$-th moment} of $μ$.
On top of that, the moments $σ(x^α)$ uniquely determine the map $σ$.

From this point of view,
the statement of \ref{thm:pronyohe} is that
the support of the finitely-supported signed measure $μ$
is already determined by \emph{finitely many} of its moments,
namely the ones that are required to construct the matrix $H_{d-1,d}$.
In fact, in this case, the weights $λ_1,…,λ_r$ can be recovered as well,
by subsequently solving a linear system of equations (cf.~\cite[Algorithm~2.1]{vdohe2017}),
so the measure $μ$ is fully determined by these moments.
The condition that $\ev_{≤d-1}$ is surjective holds if $d$ is sufficiently large,
a trivial bound being $d≥r$,
as can be seen by constructing Lagrange polynomials of degree $r-1$ for the points $ξ_1,…,ξ_r$; cf. \cite[Corollary~2.20]{vdohe2017}.
The ideal
\[
  \bigcap_{j=1}^r \idealspan{x - ξ_j}
  = \prod_{j=1}^r \idealspan{x - ξ_j}
  = \prod_{j=1}^r \idealspan{x_1 - ξ_{j1},…, x_n - ξ_{j n}}
\]
is clearly generated by polynomials of degree at most $r$,
but in the multivariate setting with $n≥2$,
unless the points $ξ_1,…,ξ_r$ are contained in a one-dimensional subspace of $\K^n$,
this bound can be much larger than necessary.
A more practical sufficient criterion for the evaluation map $\ev_{≤d-1}$ being surjective
is obtained by checking the rank of the matrix $H_{d-1,d}$.
As this rank is at most $r$,
it follows from the Vandermonde factorization in \ref{lem:vandermondefactorization:1}
that $\ev_{≤d-1}$ is surjective if and only if $\rk H_{d-1,d} = r$.

\begin{remark}
  A variation of Prony's method works with Toeplitz matrices
  of the form
  \[
    \lr{\sum_{j=1}^r λ_j ξ_j^{-α+β}}_{α,β∈ℕ^n,\,\maxdeg{α}≤d',\maxdeg{β}≤d}
  \]
  instead of Hankel matrices,
  where the moments are usually bounded in max-degree.
  For this to be defined,
  the points $ξ_1,…,ξ_r$ must have non-zero coordinates,
  so they are contained in the algebraic torus $\lr{ℂ\withoutzero}^n$.
  This is especially common when working in a trigonometric setting,
  with points on the \emph{complex torus}
  \[
    \T^n \coloneqq \{z ∈ ℂ^n \mid \abs{z_1}=\cdots =\abs{z_n} = 1\}.
  \]
  Moreover, one can work with much more general filtrations of the polynomial ring;
  see the statements in \cite[Chapter~2]{vdohe2017}.
  See also \cite{ohe2020:pronystructures}
  for an approach relating Toeplitz and Hankel matrices in this context.
\end{remark}

\section{Sesquilinearity and filtrations}\label{sec:sesquilinearity}

In this section, we set up a framework that allows us
to treat in a unified way the two different settings of
moment problems we are primarily interested in,
namely moment problems on affine space and on the torus.
See \cite[Chapter~2]{schmuedgen2017} for a similar approach
to these concepts.

\begin{definition}
  Let $R$ be a ring with a map
  $\invol{\blank}\colon R\to R$ satisfying
  \begin{align}
    \Invol{x + y} = \invol{x} + \invol{y}, &&
    \Invol{x y} = \invol{y} \invol{x}, &&
    \invol{1} = 1, &&
    \Invol{\invol{x}} = x
  \end{align}
  for all $x,y∈R$.
  Then the map $\invol{\blank}$ is called \emph{involution}
  and $R$ is an \emph{involutive ring} (also called \emph{${}^*$-ring}).
  An involutive ring $A$ with involution $\invol[A]{\blank}$
  that is also an (associative) algebra over a commutative involutive ring $R$
  is an \emph{involutive algebra} (also called \emph{${}^*$-algebra}),
  if the involution satisfies
  $\Invol[A]{r a} = \invol{r} \invol[A]{a}$
  for all $r∈R$ and $a∈A$.
  As this property means that there is no ambiguity,
  we denote the involution on $A$
  by $\invol{\blank}$ as well.
  A map $f\colon A \to A$ is \emph{$\invol{}$-semilinear} if
  $f(a+b) = f(a) + f(b)$
  and $f(r a) = \invol{r} f(a)$ holds for all $r∈R$ and $a,b∈A$.
\end{definition}

Common examples of involutive rings include the field of complex numbers $ℂ$
with complex conjugation
as well as square complex matrices with conjugate transposition as involution.
Another important example for our discussion
is given in \ref{ex:involution:conjugation} below.
Also note that any commutative ring (algebra) is an involutive ring (algebra)
with respect to the trivial involution
which leaves every element unchanged.

\begin{definition}\label{def:filtration}
  Let $\K$ be a field
  and $A$ an (associative) algebra over $\K$.
  If $F_d \subseteq A$, $d∈ℕ$, is a family of
  $\K$-vector subspaces satisfying
  \begin{itemize}
    \begin{minipage}{0.525\linewidth}
      \item $F_d \subseteq F_e$ for $d,e∈ℕ$ with $d≤e$,
      \item $A = \bigcup_{d∈ℕ} F_d$,
    \end{minipage}
    \begin{minipage}{0.4\linewidth}
      \item $1 ∈ F_0$,
      \item $F_d \cdot F_e \subseteq F_{d+e}$ for $d,e∈ℕ$,
    \end{minipage}
  \end{itemize}
  then $A$ is a \emph{filtered algebra} over $\K$
  and the family $\{F_d\}_{d∈ℕ}$ is called \emph{filtration} of $A$.
  In particular, the filtrations we consider are exhaustive.
  For simplicity of notation, we often denote
  the filtered components of the filtration by
  $A_{≤d} \coloneqq F_d$.
\end{definition}

\begin{example}\label{ex:filtrations}
  Let $\K$ be a field and $R = \K[x_1,…,x_n]$
  be the polynomial ring in $n$~variables over $\K$, for some $n∈ℕ$.
  Then the total degree of polynomials gives rise to a filtration of $R$ where
  \[\label{eq:totaldegfiltration}
    R_{≤d} = \{p ∈ R \mid \deg(p) ≤ d\}
  \]
  for $d∈ℕ$.
  Similarly, we can define a filtration $\{F_d\}_{d∈ℕ}$, on $R$
  in terms of max-degree by
  \[\label{eq:maxdegfiltration}
    F_d = \bigoplus_{\tsubstack{α∈ℕ^n\\\maxdeg{α}≤d}} \K x^α.
  \]
  Note that all the filtered components of these two filtrations
  happen to be $\K$-vector spaces of finite dimension,
  which is a useful property when it comes to computations.

  Now let $\mf{a}\subseteq R$ be an ideal with $1\notin\mf{a}$
  and define $S = R / \mf{a}$.
  If $\{F_d\}_{d∈ℕ}$, is any filtration of $R$,
  then
  $G_d \coloneqq F_d / \lr{\mf{a} \cap F_d}$
  defines a filtration of the quotient ring $S$.
  For this, observe that $G_d$ can be embedded in $G_{d+1}$
  via the injective map
  $p + \mf{a} \cap F_d \mapsto p + \mf{a} \cap F_{d+1}$,
  for all $p∈F_d$, $d∈ℕ$.
\end{example}

For the remainder of this \lcnamecref{sec:sesquilinearity},
we assume, for simplicity, that $\K$ is a field of characteristic $0$
together with an involution $\invol{\blank}$
that endows $\K$ with the structure of an involutive ring.
Moreover, we denote by $R = \K[x_1,…,x_n]$
the polynomial ring in finitely many variables
and fix a filtration $\{R_{≤d}\}_{d∈ℕ}$
that turns $R$ into a filtered algebra over $\K$
and has the property that $R_{≤d}$ is a finite-dimensional $\K$-vector space for every $d∈ℕ$.
Additionally, we assume that $R\subseteq L$
is a $\K$-subalgebra of an involutive commutative algebra $L$ over $\K$.
The involution on $L$ is denoted by $\invol{\blank}$ as well.
Typical examples are the following:
\par
\begin{example}\label{ex:involution:trivial}
      If $\K$ is any field, let $L = R$
      and define the involutions on $\K$ and $L$ to act trivially.
      The filtration $\{R_{≤d}\}_{d∈ℕ}$ on $R$ is defined by total degree as
      in \myeqref{eq:totaldegfiltration}.
      Of particular interest is the case when
      $\K$ is the field of real numbers $ℝ$ (or a subfield thereof).
\end{example}

\begin{example}\label{ex:involution:conjugation}
      If $\K$ is any field with an involution $\invol{\blank}$,
      let $L = \K[x_1^{±1},…,x_n^{±1}]$ be the ring of Laurent polynomials
      and define the involution on $L$ by
      \[
        \Invol{\sum_{α} p_α x^α} \coloneqq \sum_{α} \invol{p_α} x^{-α},
      \]
      where $p_α∈\K$, $α∈ℤ^n$,
      which turns $L$ into an involutive algebra.
      For the filtration on $R$,
      in this situation we usually pick the one that is induced by max-degree
      as in \myeqref{eq:maxdegfiltration},
      since $L$ is the coordinate ring of the algebraic torus,
      and denote it by $\{R_{≤d}\}_{d∈ℕ}$ again.

      Of particular interest is the case $\K = ℂ$ of complex numbers
      with complex conjugation as involution.
      In this case, an observation that can be significant in some applications is the following:
      If we restrict a Laurent polynomial $p∈L$ to the complex torus $\T^n$,
      then the involution $\invol{p}$ is the complex conjugate of $p$ as a function on $\T^n$,
      so we have
      \[
        \invol{p}(ξ) = \conj{p(ξ)}
      \]
      for all $ξ∈\T^n$, since $ξ^{-α} = \conj{ξ}^α$ for all $α∈ℤ^n$.
      In particular, the Laurent polynomial
      $p$ is a real function on $\T^n$
      if and only if $\invol{p} = p$,
      \ie $p_α = \conj{p_{-α}}$ for all $α$.
      Furthermore, note that,
      if $\mf{a} \subseteq L$ is a vanishing ideal of a set contained in $\T^n$,
      then it follows that $\invol{\mf{a}} = \mf{a}$.
\end{example}

\begin{definition}\label{def:sesquilinearform}
  Let $σ\colon L \to \K$ be a $\K$-linear map.
  Then we define the $\K$-sesquilinear form
  \[
    \sform[σ]{\blank,\blank}\colon L × L \longrightarrow \K,\quad (q, p) \longmapsto σ(\invol{q} p),
  \]
  which is $\invol{}$-semilinear in the first and linear in the second argument.
  Defining sesquilinear forms to be semilinear in the first
  rather than in the second argument is an arbitrary choice.
  We choose this convention as it simplifies our notation later on.
  By restriction, we can also view this as a sesquilinear form on $R$
  as well as on the finite-dimensional vector spaces $R_{≤d}$, $d∈ℕ$.
  Note that this is a symmetric bilinear form if the involution is trivial.

  A form $\sform{\blank,\blank}$ on a $\K$-vector space $U$
  is \emph{Hermitian} if
  $\sform{q,p} = \invol{\sform{p,q}}$ for all $p,q∈U$.
  If the involution is trivial, as in \ref{ex:involution:trivial},
  then this always holds for $\sform[σ]{\blank,\blank}$,
  as the form is symmetric in that case.
  When $\K$ is (a subfield of) the complex numbers $ℂ$,
  then a Hermitian form $\sform[σ]{\blank,\blank}$ on $U$
  is \emph{positive-semidefinite} if, additionally,
  $\sform[σ]{p,p} ≥ 0$ for all $p∈U$.
  Note that this never holds if
  $\K \nsubseteq ℝ$ and
  the involution is linear,
  rather than $\invol{}$-semilinear,
  unless the form is trivial.
\end{definition}

\begin{remark}\label{rem:monomialbasis:hankel:toeplitz}
  Assume that a family of monomials $\{x^α\}_{α∈J} \subseteq R_{≤d}$
  for a suitable index set $J\subseteq ℕ^n$
  forms a basis of the finite-dimensional vector space $R_{≤d}$
  and that the involution $\invol{\blank}$ is trivial.
  Then the Gramian matrix of $\sform[σ]{\blank,\blank}$
  with respect to this basis is of the form
  \[
    \lr{\sform[σ]{x^α,x^β}}_{α,β∈J}
    = \lr{σ\lr{x^{α+β}}}_{α,β∈J},
  \]
  which is a (generalized) Hankel matrix.

  Likewise, if $\{x^α\}_{α∈J} \subseteq R_{≤d}$ is a basis of $R_{≤d}$,
  but $L$ is the ring of Laurent polynomials
  with involution $\invol{\blank}\colon L\to L$
  defined as in \ref{ex:involution:conjugation},
  then the Gramian matrix with respect to this basis
  is of the form
  \[
    \lr{\sform[σ]{x^α,x^β}}_{α,β∈J}
    = \lr{σ\lr{x^{-α+β}}}_{α,β∈J},
  \]
  which is a (generalized) Toeplitz matrix.
\end{remark}

\begin{lemma}\label{lem:inducedformonquotient}
  Assume that $σ\colon L \to \K$ is a $\K$-linear map,
  $\mf{a}\subseteq L$ is an ideal such that $\mf{a},\invol{\mf{a}} \subseteq \kernel σ$.
  Let $W\subseteq L$ be a $\K$-vector subspace.
  Then the sesquilinear form $\sform[σ]{\blank,\blank}$ on $L$
  induces a sesquilinear form
  \[
    \submodulequotient{W}{\mf{a}} × \submodulequotient{W}{\mf{a}} \longrightarrow \K,\qquad
    (\resid{q}, \resid{p}) \longmapsto \sform[σ]{q, p} = σ(\invol{q} p).
  \]
\end{lemma}
Here, $\resid{q},\resid{p}$ denotes the residue class of
polynomials $q,p∈W$ modulo $\mf{a}\cap W$.
We denote the induced sesquilinear form on $\submodulequotient{W}{\mf{a}}$
by $\sform[σ]{\blank,\blank}$ again.
Also note that
that the requirements $\mf{a}\subseteq \kernel σ$ and $\invol{\mf{a}}\subseteq \kernel σ$
are equivalent when the sesquilinear form $\sform[σ]{\blank,\blank}$ on $L$ is Hermitian.
\begin{proof}
  Let $p,q∈W$.
  If $p ∈ \mf{a} \cap W$, then $\invol{q}p$ is contained in
  $\mf{a}\subseteq \kernel σ$, so $σ(\invol{q}p) = 0$.
  Likewise, if $q ∈ \mf{a} \cap W$, then $\invol{q} p ∈ \invol{\mf{a}} \subseteq \kernel σ$,
  so the sesquilinear form on $\submodulequotient{W}{\mf{a}}$ is well-defined.
\end{proof}

\begin{remark}\label{def:inducedmaponquotient}
  If $σ\colon L\to \K$ is $\K$-linear and $\mf{a}\subseteq L$ is an ideal
  such that $\mf{a}\subseteq\kernel σ$,
  then the sesquilinear form $\sform[σ]{\blank,\blank}$ on $L$
  does \emph{not} induce a sesquilinear form on the quotient spaces
  $\submodulequotient{W}{\mf{a}}$.
  (Observe that this would need $\invol{\mf{a}}\subseteq \kernel σ$
  or require the form to be Hermitian, as in \ref{lem:inducedformonquotient}.)
  Many of our arguments here can be transferred to this setting
  by working with a sesquilinear \emph{map} instead of a sesquilinear form;
  for details we refer to \cite[Definition~3.1.12]{wageringel2021}.
\end{remark}

\section{Factorization properties}\label{sec:factorization}

The Vandermonde factorization of \ref{lem:vandermondefactorization:1}
is an essential aspect of Prony's method.
Here, we analyze how to transfer it from measures on zero-dimensional
to measures on positive-dimensional algebraic varieties.
The statements here are also motivated
by the study of finite-rank Hankel operators as in \eg{}\cite{mourrain17:polyexp}.
In the positive-dimensional setting, such operators are not of finite rank anymore,
but some properties are still valid.

Let $\K, R, L$ be as in \ref{sec:sesquilinearity},
so $\K$ is a field of characteristic~$0$, $R = \K[x_1,…,x_n]$ is the polynomial ring in $n$ variables
endowed with a filtration $\{R_{≤d}\}_{d∈ℕ}$
and $L$ is an involutive commutative $\K$-algebra such that $R\subseteq L$.

We wish to examine more closely the following situation.
Let $\mf{a}\subseteq L$ be an ideal and
let $σ\colon L\to\K$ be a $\K$-linear map
with the property that $\mf{a} \subseteq \kernel σ$.
This means that the map $σ$ factors via
the quotient homomorphism
\[
  \proj{\mf{a}}\colon L \longrightarrow L/\mf{a},\qquad
  p \longmapsto \resid{p} \coloneqq p + \mf{a},
\]
which we denote by $\proj{\mf{a}}$,
and a $\K$-linear map $\resid{σ}\colon L/\mf{a} \to \K$, denoted by $\resid{σ}$.

\begin{example}\label{eg:reduced0dim}
  Assume that $L$ is the polynomial ring $R$ and $ξ∈\K^n$
  (or that $L$ is the Laurent polynomial ring in $n$ variables and $ξ∈\lr{\K\withoutzero}^n$).
  Then,
  for the maximal ideal $\maxideal{ξ} = \idealspan{x-ξ} \subseteq L$,
  this gives the evaluation homomorphism at the point $ξ$,
  \[
    \proj{\maxideal{ξ}}\colon L \longrightarrow L/\maxideal{ξ} \cong \K,\qquad
    x^α \longmapsto \resid{x}^α = ξ^α,
  \]
  for $α∈ℕ^n$ (or $α∈ℤ^n$),
  so $\proj{\maxideal{ξ}}(p) = p(ξ)$ for $p∈L$.
  Note further that, for any $\K$-linear map $σ\colon L\to \K$ with $\maxideal{ξ}\subseteq\kernel σ$,
  the linear map $\resid{σ}\colon L/\maxideal{ξ}\cong \K \to \K$
  is determined by a single scalar $λ∈\K$, with respect to a suitable basis.
  Thus, $σ = λ\proj{\maxideal{ξ}} = λ\ev_{ξ} ∈ \kdual L$,
  which we can interpret as an exponential sum of rank~$1$
  if $λ≠0$ (cf.~\ref{sec:multivarprony}).

  More generally, consider the zero-dimensional ideal
  $\mf{a} = \bigcap_{j=1}^r \maxideal{ξ_j}$,
  for distinct points $ξ_1,…,ξ_r$.
  Then it follows from the Chinese Remainder Theorem (cf.~\cite[Chapter~2.1.2, Proposition~5]{bourbaki:comalg}) that
  \[
    L/\mf{a}
    \cong \bigoplus_{j=1}^r L/\maxideal{ξ_j}
    \cong \K^r,
  \]
  where $\proj{\mf{a}}(p)$ is identified with $\lr{p(ξ_1),…,p(ξ_r)}$ for $p∈L$.
  As a $\K$-linear map with respect to the monomial basis of $L$,
  we can view $\proj{\mf{a}}$ as being described by an infinite Vandermonde matrix
  associated to the points $ξ_1,…,ξ_r$.
  If $σ\colon L\to \K$ is a $\K$-linear map with $\mf{a}\subseteq \kernel σ$,
  then it is of the form $σ = \sum_{j=1}^r λ_j \ev_{ξ_j}$
  with suitable parameters $λ_1,…,λ_r∈\K$,
  which corresponds to an exponential sum of rank~$r$ if $λ_1,…,λ_r≠0$.
\end{example}

The ideal $\mf{a}$ does not need to be radical in this setup.
An explicit example is given in \cite[Example~3.2.2]{wageringel2021};
more generally polynomial exponential series as studied in \cite{mourrain17:polyexp}
correspond to non-radical ideals.
Later on, we will focus on the case in which $\mf{a}$ is a vanishing ideal, though.

As $R$ is endowed with a filtration $\{R_{≤d}\}_{d∈ℕ}$
for which each component $R_{≤d}$ is finite-dimensional
and since $R\subseteq L$,
we can restrict the map $\proj{\mf{a}}\colon L\to L/\mf{a}$
to a map on finite-dimensional vector subspaces
$R_{≤d} \to R_{≤d}/\lr{\mf{a}\cap R_{≤d}}$,
which we denote by $\evd{\mf{a}}{d}$,
as explained in \ref{ex:filtrations}.

An important ingredient of Prony's method
is that we can extract information about the vanishing ideal
from the kernel of the moment matrix, if the moment matrix is sufficiently large;
see \ref{lem:vandermondefactorization:2}.
In the following, we examine what is required to transfer this property
to the setting of ideals which are possibly not of dimension zero,
but are of higher dimension.
This is answered by the following \lcnamecref{thm:hankelopfactorization}
as well as \ref{lem:hankelkernelinjectivity} below.

\begin{theorem}\label{thm:hankelopfactorization}
Let $\mf{a}\subseteq L$ be an ideal and
let $σ\colon L\to\K$ be a $\K$-linear map with $\mf{a}\subseteq \kernel σ$.
Then the $\K$-linear map
\[
  H\colon R \longrightarrow \semikdual R,\qquad
  p \longmapsto \lr{q \mapsto \sform[σ]{q, p}},
\]
factors as
\begin{equation}
  \label{eq:hankelopfactors}
  \begin{tikzcd}[row sep=tiny,ampersand replacement=\&]
    R \arrow[r, "\proj{\mf{a}}"] \&
    \submodulequotient{R}{\mf{a}} \arrow[r] \&
    \semikdual{\submodulequotient{R}{\invol{\mf{a}}}} \arrow[r, "\Transp{\proj{\invol{\mf{a}}}}"] \&
    \semikdual R,\\
    \& p + \mf{a}\cap R \arrow[r, mapsto] \&
    \lr{q + \invol{\mf{a}}\cap R \mapsto \sform[σ]{q, p}},
  \end{tikzcd}
\end{equation}
where $\Transp{\proj{\invol{\mf{a}}}}(φ) = φ \compose \proj{\invol{\mf{a}}}$
for $φ ∈ \semikdual{\submodulequotient{R}{\invol{\mf{a}}}}$.

Moreover, the truncated map between finite-dimensional vector subspaces given by
\[
  H_{d+δ,d}\colon
  R_{≤d} \longrightarrow \semikdual{R_{≤d+δ}},\qquad
  p \longmapsto \lr{q \mapsto \sform[σ]{q, p}},
\]
for $d,δ∈ℕ$,
factors as
\begin{equation}
  \label{eq:hankelmatfactors}
  \begin{tikzcd}[row sep=tiny,ampersand replacement=\&]
    R_{≤d} \arrow[d, "\evd{\mf{a}}{d}"] \arrow[r, "H_{d+δ,d}"] \&
    \semikdual{R_{≤d+δ}} \\[5ex]
    \submodulequotient{R_{≤d}}{\mf{a}} \arrow[r, "\resid{H_{d+δ,d}}"] \&
    \semikdual{\submodulequotient{R_{≤d+δ}}{\invol{\mf{a}}}} \arrow[u, "\Transp{\evd{\invol{\mf{a}}}{d+δ}}"],\\
    p + \mf{a}\cap R_{≤d} \arrow[r, mapsto] \&
    \lr{q + \invol{\mf{a}}\cap R_{≤d+δ} \mapsto \sform[σ]{q, p}}.
  \end{tikzcd}
\end{equation}
\end{theorem}
\begin{proof}
  Due to the inclusion $\mf{a}\subseteq \kernel σ$,
  we have that
  \[
    σ\lr{\Invol{q + \invol{\mf{a}} \cap R} \lr{p + \mf{a} \cap R}}
    = σ\lr{\lr{\invol{q} + \mf{a} \cap \invol{R}} \lr{p + \mf{a} \cap R}}
    = σ\lr{\invol{q} p}
    = \sform[σ]{q, p},
  \]
  for all $q,p∈R$,
  which shows the first factorization property.
  The other one follows analogously.
\end{proof}

The truncated map $H_{d+δ,d}$ is of importance for us,
since we are interested in recovery from finitely many moments.
By \ref{thm:hankelopfactorization},
it holds that
\[
  \mf{a}\cap R_{≤d} \subseteq \kernel H_{d+δ,d}
\]
and we ask when this is an equality.
This leads to the following \lcnamecref{lem:hankelkernelinjectivity}.
\par
\begin{corollary}\label{lem:hankelkernelinjectivity}
  If the map $\resid{H_{d+δ,d}}\colon
  \submodulequotient{R_{≤d}}{\mf{a}} \to \semikdual{\submodulequotient{R_{≤d+δ}}{\invol{\mf{a}}}}$
  is injective, then
  \[
    \kernel\lr{H_{d+δ,d}} =
    \kernel\lr{\evd{\mf{a}}{d}} =
    \mf{a} \cap R_{≤d}.
  \]
\end{corollary}
\begin{proof}
  Due to the factorization \myeqref{eq:hankelmatfactors}
  and since the map
  $\Transp{\evd{\invol{\mf{a}}}{d+δ}}$ is injective,
  the equality holds if and only if the map
  $\resid{H_{d+δ,d}}$ is injective.
\end{proof}

As the vector space dimension of the codomain of $\resid{H_{d+δ,d}}$
is finite and at least as large as the dimension of the domain,
saying that $\resid{H_{d+δ,d}}$ is injective
is the same as saying that the map $\resid{H_{d+δ,d}}$ has full rank.
As such, this can be regarded as a variant of
the statement about the Vandermonde factorization in
\ref{lem:vandermondefactorization}.

\begin{remark}
  In this formalism, $\evd{\mf{a}}{d}$ is always surjective,
  which is an important difference from the Vandermonde factorization
  considered in \ref{lem:vandermondefactorization:1},
  as the Vandermonde matrix considered there can be non-surjective for small $d$.
  This is explained further in
  \cref{ex:hankelmatfactors:points} below.
  There, for an ideal of the form
  $\mf{a} = \bigcap_{j=1}^r \maxideal{ξ_j}$,
  the dimension of
  $\image\lr{\evd{\mf{a}}{d}} = \submodulequotient{R_{≤d}}{\mf{a}}$ as vector space is at most $r$,
  but can be smaller.
  Equality holds if and only if the corresponding Vandermonde matrix has rank $r$,
  which only holds if $d$ is sufficiently large.
\end{remark}

Moreover, we remark that the map $\resid{H_{d+δ,d}}$ is injective in particular when
$σ$ is a moment functional of a measure and
$\mf{a}$ is the vanishing ideal of its support,
as will be shown in \ref{thm:idealequalkernel}.

\begin{example}\label{ex:hankelmatfactors:points}
  Let us revisit \ref{eg:reduced0dim},
  so let $\mf{a} \coloneqq \bigcap_{j=1}^r \maxideal{ξ_j} \subseteq L$
  for distinct points $ξ_1,…,ξ_r ∈ \K^n$,
  where now we assume that $L = R = \K[x_1,…,x_n]$
  is endowed with the trivial involution and the filtration induced by total degree.

  If $d$ is sufficiently large,
  $\evd{\mf{a}}{d}$ has rank $r$
  and we have
  $\submodulequotient{R_{≤d}}{\mf{a}} \cong \bigoplus_{j=1}^r R / \maxideal{ξ_j} \cong \K^r$.
  Hence, we also have
  $\submodulequotient{R_{≤d+δ}}{\mf{a}} \cong \K^r$ for all $δ∈ℕ$.
  If $σ\colon R\to \K$ is a $\K$-linear map with $\mf{a}\subseteq \kernel σ$,
  then, by \ref{eg:reduced0dim},
  it is of the form $σ = \sum_{j=1}^r λ_j \ev_{ξ_j}$
  for some $λ_1,…,λ_r ∈ \K$.
  Thus, the map $\resid{H_{d+δ,d}}$ corresponds to the diagonal matrix
  $\diag\lr{λ_1,…,λ_r}$ with respect to the natural bases.
  Clearly, it is injective if and only if $λ_1,…,λ_r≠0$,
  which illustrates the connection of \ref{lem:hankelkernelinjectivity}
  to \ref{lem:vandermondefactorization:2}.
\end{example}
\par
Although for zero-dimensional ideals as in the preceding example
it is enough to consider the case $δ=0$ to infer that
$\kernel H_{d+δ,d} = \mf{a} \cap R_{≤d}$
if $d$ is sufficiently large,
this does not hold in general (cf.~\ref{ex:pointskernelunequaltruncatedideal}).
We will see a non-trivial example in
\ref{ex:signedmeasurewrongkernel},
which involves an ideal of positive dimension.
In connection to that, \ref{thm:idealequalkernelextended}
will show that it can be useful to consider $δ$ larger than $0$.

\section{Recovery of the support from moments}\label{sec:support}

In this \lcnamecref{sec:support},
we explore how to recover the underlying algebraic variety
that a measure is supported on, by using finitely many of its moments.
We consider a non-negative or signed measure $μ$ whose support
lives in the affine space $ℝ^n$ or the complex torus $\T^n$
and wish to find the smallest variety that contains the support.
Following the notation of \ref{sec:sesquilinearity},
we consider the following two cases,
to which we also refer as \emph{affine} and \emph{trigonometric} cases,
respectively:
\begin{enumerate}
  \item
    $Ω = ℝ^n$,
    $\K = ℝ, L = R = ℝ[x_1,…,x_n]$ with trivial involutions
    (cf.~\ref{ex:involution:trivial});
  \item
    $Ω = \T^n$,
    $\K = ℂ, R = ℂ[x_1,…,x_n], L = ℂ[x_1^{±1},…,x_n^{±1}]$
    with complex conjugation and
    involution $\invol{\blank}$ on $L$
    defined as in \ref{ex:involution:conjugation}.
\end{enumerate}

Additionally, we fix a filtration $\{R_{≤d}\}_{d∈ℕ}$ of $R$
consisting of finite-dimensional vector spaces.
Recall that the support of a non-negative or signed measure
is defined as follows.

\begin{definition}[{cf.~\cite[Chapter~1.3]{schwartz1973}}]
  Let $μ$ be a signed measure on $Ω$.
  Then
  \[
    \supp μ \coloneqq
    \braced*{ξ∈Ω \mid μ\restrict{U} ≠ 0\text{\ for all open neighborhoods $U\subseteq Ω$, $ξ∈U$}}
  \]
  is called \emph{support} of $μ$,
  where $μ\restrict{U}$ denotes the restriction of $μ$ to $U$.
\end{definition}

By convention, we consider the support in terms of the standard topology on $Ω$.
The complement of $\supp μ$ in $Ω$ is the union of all open sets
on which $μ$ is constantly zero and is open,
so $\supp μ$ is a closed set.
When we consider the support in terms of the Zariski topology,
we denote it by $\zariski{\supp μ}$ (as a subset of $Ω$ or $(ℂ\withoutzero)^n$).
It is the smallest Zariski-closed set containing $\supp μ$.

This topic has been studied in various forms,
usually in the real affine case with non-negative measures
(\eg{}\cite{lasserre2015:algebraicexponential,lasserre2021:empiricalmomentschristoffel})
and an emphasis on finitely-supported measures;
see for instance \cite{laurentrostalski2012}.
The case of plane algebraic curves has also been investigated in
\cite{vetterli2016}, with a focus on the presence of noise.
The case of plane trigonometric curves on the torus has been considered in
\cite{ongie15:piecewisesmooth,ongie2016:piecewiseconstant}.

We unify the different noise-free settings in \ref{thm:idealequalkernel} and
expand the existing results by \ref{thm:idealequalkernelextended},
a statement for compactly-supported signed measures,
as well as \ref{cor:idealequalkernelpolynomial,thm:idealequalkernel:mixture}.
Moreover, we give examples that highlight the differences
between signed and non-negative measures.

\subsection{Signed measures}\label{sec:support:signedmeasures}

Here, we consider a signed measure $μ$ on $Ω$.
If $\K = ℂ$, as in the trigonometric case, then $μ$ is a complex measure.
As a consequence of the Riesz representation theorem
(see \eg{}\cite[Theorem~6.19]{rudin1987}),
these measures can be defined as elements
in the continuous dual space of
the space $\contincompact{0}(Ω)$ of compactly-supported continuous functions from $Ω$ to $\K$.
We refer to \cite[Chapter~1.2]{schwartz1973} for an extensive treatment of this topic.

In the trigonometric case, all the moments of $μ$ are defined,
as the torus $\T^n$ is compact.
In order to speak of moments
$\int_Ω x^α \d μ$, $α∈ℕ^n$, in the affine case,
we need to make additional assumptions on the measure $μ$,
since the monomials $x^α$ are not compactly-supported functions on $ℝ^n$.
Certainly, the moments are defined when the measure $μ$ itself is compactly supported.
More generally, all the moments are defined
for signed measures with a sufficiently rapid decay toward infinity,
such as those that can be written as a product $μ = g μ_0$
of a Schwartz function $g$ and a tempered distribution $μ_0$
(see \eg{}\cite[Chapter~2]{grafakos2014} or \cite[Chapter~7]{schwartz1973}),
which in particular includes Gaussians and mixtures thereof.
In this \lcnamecref{sec:support:signedmeasures},
we focus on signed measures with compact support only,
as these are determined by their moments.

First, let us take note of the following elementary properties of
the support of the product between a measure and a continuous function.
\par
\begin{lemma}\crefalias{enumi}{lemma}
  Let $μ$ be a signed measure on $Ω$ and let $f,g ∈ \contin{0}(Ω)$ be continuous functions.
  Then:
  \begin{thm-enumerate}
    \item\label{lem:nonvanishingsupportinclusion}
      $\nonV{f} \cap \supp μ \subseteq \supp\lr{f μ}$,
      where $\nonV{f} \subseteq Ω$ denotes the set of points
      in which $f$ does not vanish.
    \item\label{lem:zeroonsupport}
      The measure $f μ$ is zero if and only if $f$ vanishes on $\supp μ$.
    \item\label{lem:samevanishingonsupport}
      If $\nonV{f} \cap \supp μ = \nonV{g} \cap \supp μ$,
      then $\supp\lr{f μ} = \supp\lr{g μ}$.
  \end{thm-enumerate}
\end{lemma}
\begin{proof}
  For \localref{lem:nonvanishingsupportinclusion},
  let $ξ ∈ \supp μ$ be any point such that $f(ξ) ≠ 0$
  and let $U\subseteq Ω$ be an arbitrary open neighborhood of~$ξ$.
  We need to show that $f μ\restrict{U} ≠ 0$.
  For this, let $U_0 \subseteq U$ be an open neighborhood of~$ξ$
  in which $f$ does not have any roots.
  Since $ξ$ is a support point of $μ$,
  there exists a compactly-supported continuous function $φ ∈ \contincompact{0}(U_0)$
  such that $\int_{U_0} φ\,\d μ ≠ 0$.
  Then
  $ψ \coloneqq \frac{φ}{f} ∈ \contincompact{0}(U_0)$
  can be extended trivially to a compactly-supported function $ψ ∈ \contincompact{0}(U)$
  and we have $\int_U ψ\,\d (f μ) = \int_{U_0} \frac{φ}{f}\,\d (f μ) = \int_{U_0} φ\,\d μ ≠ 0$
  and thus $f μ\restrict{U} ≠ 0$,
  which proves the statement.
  For part \localref{lem:zeroonsupport}, assume that $f μ$ is zero.
  Then $\supp\lr{f μ} = ∅$,
  so $f$ vanishes on $\supp μ$ by \localref{lem:nonvanishingsupportinclusion}.
  The converse holds by \cite[Chapter~3, Theorem~33, addendum]{schwartz1973}.
  Finally, for part \localref{lem:samevanishingonsupport},
  observe that the complement of $\supp\lr{f μ}$
  consists of the union of all open sets $U\subseteq Ω$ satisfying $f μ\restrict{U} = 0$.
  By \localref{lem:zeroonsupport}, this is equivalent to $f$ vanishing on $\supp\lr{μ\restrict{U}}$.
  By hypothesis, this is the case if and only if $g$ vanishes on $\supp\lr{μ\restrict{U}}$,
  which in turn is equivalent to $g μ\restrict{U} = 0$
  and thus completes the proof.
\end{proof}

For the remainder of this \lcnamecref{sec:support:signedmeasures},
we fix a filtration $\{L_{≤d}\}_{d∈ℕ}$ of $L$
for which all the components are finite-dimensional vector spaces.
In the affine case, we may choose $L_{≤d} = R_{≤d}$.
Additionally, we denote by $B_d^L$ and $B_d^R$
any bases of the filtered components $L_{≤d}$ and $R_{≤d}$, respectively.
With this notation, we arrive at the following \lcnamecref{thm:idealequalkernelextended}.
\par
\begin{theorem}\label{thm:idealequalkernelextended}
  Let $μ$ be a compactly-supported signed measure on $Ω$,
  denote by $\mf{a} \coloneqq \Id{\supp μ} \subseteq L$
  the vanishing ideal of (the Zariski closure of) its support
  and let $σ\colon L\to\K$ be its moment functional.
  Let $d∈ℕ$. Then
  \[
    \mf{a} \cap R_{≤d} = \kernel H_{d',d}
  \]
  holds for all sufficiently large $d'∈ℕ$,
  where
  $H_{d',d} \coloneqq \lr{\sform[σ]{w,v}}_{w ∈ B_{d'}^L,\,v ∈ B_d^R}$.
\end{theorem}
It then follows from Hilbert's basis theorem that
$\mf{a}$ is generated by $\kernel H_{d',d}$
if $d∈ℕ$ is sufficiently large.
\begin{proof}
  As the measure $μ$ is compactly supported, all its moments exist.
  Let $d∈ℕ$ be arbitrary
  and observe that
  \[\label{eq:idealsubsetkernel}
    \mf{a}\cap R_{≤d} \subseteq \kernel H_{d',d}
    = \braced*{p∈R_{≤d}\mid \sform[σ]{q,p} = 0\text{\ for all $q∈L_{≤d'}$}},
  \]
  for all $d'∈ℕ$.
  Indeed, if $p∈\mf{a}$, then $p$ vanishes on the support of $μ$,
  so $\sform[σ]{q,p} = \int_Ω \invol{q} p\,\d μ = 0$
  for all $q∈L$, by \ref{lem:zeroonsupport}.
  More specifically, we have a descending chain
  \[
    R_{≤d} \supseteq \kernel H_{0,d} \supseteq \kernel H_{1,d} \supseteq \cdots \supseteq \mf{a}\cap R_{≤d}
  \]
  which must stabilize,
  so we can fix a $d'∈ℕ$ such that
  \[\label{eq:kernelstabilizes}
    \kernel H_{d',d} = \kernel H_{d'+δ,d}
  \]
  holds for all $δ∈ℕ$.

  Assume that $\kernel H_{d',d} \nsubseteq \mf{a}\cap R_{≤d}$.
  Then we can choose a polynomial $p ∈ \kernel H_{d',d}$
  with $p\notin\mf{a}$,
  so $p$ does not vanish everywhere on $\supp μ$.
  Hence, by \ref{lem:zeroonsupport},
  the signed measure $ν \coloneqq p μ$ is non-zero,
  so there exists a compactly-supported continuous function
  $φ ∈ \contincompact{0}(Ω)$ such that
  $\int_Ω φ\,\d ν ≠ 0$.
  By the Weierstrass approximation theorem
  (see \cite[Chapter~5, Theorem~8.1]{conway1990} for the affine real\footnote{%
    This argument would not hold if, in the affine case,
    we were to work over the field of complex numbers,
    as the algebra of polynomials on $ℂ^n$ is not closed under conjugation.}
  and \cite[Corollary~3.2.2]{grafakos2014}
  for the trigonometric version),
  the function $φ$ can be uniformly approximated by polynomials in $L$
  on a compact set containing the support of the measure $ν$,
  which implies that not all moments of $ν$ can be zero.
  Hence, there exists a polynomial $q∈L$ such that
  $\int_Ω q\,\d ν = \int_Ω q p\,\d μ = \sform[σ]{\invol{q},p} ≠ 0$.
  As $\invol{q}∈L_{≤d'+δ}$ for some $δ∈ℕ$,
  this implies that $p \notin \kernel H_{d'+δ,d}$,
  which is a contradiction to \myeqref{eq:kernelstabilizes},
  by the choice of the polynomial $p$.
\end{proof}

\begin{remark}\label{rem:compactsupportdiscussion}
  In the proof of \ref{thm:idealequalkernelextended},
  the hypothesis that the support of the signed measure $μ$ is compact
  does not only guarantee that all its moments exist,
  but, more importantly,
  it asserts that the signed measure $ν = p μ$ is determined by its moments,
  so that $ν$ is already zero if all its moments vanish.
  This does not in general hold for measures that are not compactly supported --
  not even for rapidly decreasing functions.
  For instance, let $g$ be a non-zero Schwartz function on $ℝ^n$
  such that all its derivatives vanish at the origin,
  \ie $(\partial^α g)(0) = 0$ for all $α∈ℕ^n$.
  Then its Fourier transform $\fourier{g}$ is a non-zero Schwartz function satisfying
  \[
    (-1)^n \lr{2π\I}^{\totaldeg{α}} \int_{ℝ^n} x^α \fourier{g}(x) \d x
    = \lr{\partial^α g}(0)
    = 0
  \]
  for all $α∈ℕ^n$ (cf.~\cite[Proposition~2.2.11\,(10)]{grafakos2014}),
  so all the moments of $\fourier{g}$ are zero.
\end{remark}

\begin{remark}\label{rem:idealequalkernelextended}
  In the affine case of \ref{thm:idealequalkernelextended},
  we can choose the filtration of $L$ as $L_{≤d} = R_{≤d}$ for all $d∈ℕ$.
  Let $H_{d',d}$ be
  the rectangular moment matrix satisfying
  the statement of the \lcnamecref{thm:idealequalkernelextended},
  so $\mf{a} \cap R_{≤d} = \kernel H_{d',d}$.
  By \ref{lem:hankelkernelinjectivity},
  this equality can only hold when the induced map
  $\resid{H_{d',d}}$ on the quotient spaces,
  as in \myeqref{eq:hankelmatfactors} of \ref{thm:hankelopfactorization},
  is injective.
  This implies $d' ≥ d$, for this choice of filtration in the affine case.
  This is in contrast to the statement of \ref{thm:pronyohe}
  in which the moment matrix $H_{d',d}$ had a different shape.
\end{remark}

By \ref{thm:idealequalkernelextended},
we can recover the vanishing ideal of the support
from finitely many moments.
In particular, this means that the kernel of the non-truncated moment map
also yields the vanishing ideal,
as the following statement shows.
\par
\begin{corollary}\label{cor:idealequalkernelnontruncated}
  Under the assumptions of \ref{thm:idealequalkernelextended}, we have
  \[
    \mf{a} \cap R = \kernel H,
  \]
  where $H$ denotes the map
  $H\colon R \to \semikdual{L}$, $p\mapsto \lr{q\mapsto \sform[σ]{q,p}}$.
\end{corollary}
\begin{proof}
  To see this, first observe that we always have
  the inclusion $\mf{a} \cap R \subseteq \kernel H$,
  by \ref{lem:zeroonsupport}.
  On the other hand, if $p∈\kernel H$, then $p∈R_{≤d}$ for some $d∈ℕ$.
  In particular, this implies $\sform[σ]{q,p} = 0$ for all $q∈L_{≤d'}\subseteq L$
  and arbitrary $d'∈ℕ$.
  Choosing $d'$ as in \ref{thm:idealequalkernelextended},
  we therefore obtain
  $p∈\kernel H_{d',d} = \mf{a}\cap R_{≤d}$,
  so the statement follows.
\end{proof}

\begin{remark}\label{rem:idealequalkernelextended:nobound}
  \ref{thm:idealequalkernelextended} does not quantify
  what it means for $d'∈ℕ$ to be large enough for the statement to hold.
  In general, the choice of $d'$ cannot be made purely based on knowledge
  of the support or its vanishing ideal,
  but it must inherently depend on the signed measure itself.
  Indeed, for arbitrarily large $d,d'∈ℕ$,
  one can construct a signed measure with the following properties:
  its support is compact and Zariski-dense, so its vanishing ideal is zero,
  and all the low order moments vanish so that the matrix
  $H_{d',d} = \lr{\sform[σ]{w,v}}_{w ∈ B_{d'}^L,\,v ∈ B_d^R}$
  is zero.
  Hence, the kernel of $H_{d',d}$ is non-zero
  and thus is not a generating set of the zero ideal,
  the vanishing ideal of the support.
  In other words, $d'$ is not large enough
  for the statement of the \lcnamecref{thm:idealequalkernelextended} to hold.
  However, for particular signed measures,
  a bound on $d'$ is given in \ref{cor:idealequalkernelpolynomial}.
\end{remark}

\begin{remark}\label{rem:idealequalkernelextended:bases}
  In the trigonometric case,
  we could also state \ref{thm:idealequalkernelextended}
  in a more symmetric fashion
  in terms of a matrix for which both rows and columns are indexed
  by $B_d^L$, a basis of the filtered component $L_{≤d}$.
  We prefer to index the columns by $B_d^R$ instead
  because it allows for a finer filtration,
  \ie the filtered components $R_{≤d}$ can be chosen to be of smaller dimension
  than the components $L_{≤d}$,
  and every ideal in $L$ can be generated by elements in $R$.
  Indexing the rows of the matrix by $B_{d'}^L$
  is needed in the proof of \ref{thm:idealequalkernelextended}
  due to the use of the Weierstrass approximation theorem.
  This leads to the question whether a statement
  similar to \ref{thm:idealequalkernelextended} is possible
  in which rows and columns are indexed by $B_{d'}^R$, $B_d^R$,
  \ie bases of components of the filtration on $R$
  instead of $L$.
  In general,
  this is answered negatively by the following example,
  but a positive answer is possible for non-negative measures,
  as will be shown in \ref{sec:support:nonnegativemeasures}.
\end{remark}
\par
\begin{example}\label{ex:signedmeasurewrongkernel}
  We consider the two-dimensional trigonometric case,
  so let $n=2$.
  Let $v_1 \coloneqq \lr{2,1}, v_2 \coloneqq \lr{1,2} ∈ ℤ^2$ and
  define the functionals
  \[
    σ_j\colon L\longrightarrow ℂ,\qquad
    x^α \longmapsto
    \begin{cases}
      1 &\text{if $\scalarprod{α}{v_j} = 0$,}\\
      0 &\text{otherwise,}
    \end{cases}
  \]
  for $α∈ℤ^2$ and $j=1,2$.
  These are moment functionals of uniform measures supported on
  the one-dimensional varieties in $\T^2$ that are defined by the polynomials
  $x_1 - x_2^2$ and $x_1^2 - x_2$, respectively,
  and are depicted in \ref{fig:trigonometriclines}.
  Thus, the functional $σ \coloneqq σ_1 - σ_2$
  is a moment functional of a signed measure.

  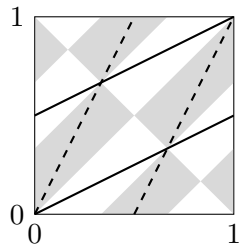
\begin{figure}[ht]
    \centering
    \begin{tikzpicture}
      \begin{axis}[
        xmin=0, xmax=1,
        ymin=0, ymax=1,
        xtick distance = 1,
        ytick distance = 1,
        minor x tick num = 3,
        minor y tick num = 3,
        tick style={draw=none},
        width = 4.2cm,
        height = 4.2cm,
        ]
        \fill[opacity=.3,gray] (2/3,0) -- (5/6,1/6) -- (2/3,1/3) -- (1/3,0) -- cycle;
        \fill[opacity=.3,gray] (1/3,1) -- (1/6,5/6) -- (1/3,2/3) -- (2/3,1) -- cycle;
        \fill[opacity=.3,gray] (0,0) -- (1/2,1/2) -- (1/3,2/3) -- (0,1/3) -- cycle;
        \fill[opacity=.3,gray] (1,1) -- (1/2,1/2) -- (2/3,1/3) -- (1,2/3) -- cycle;
        \fill[opacity=.3,gray] (0,2/3) -- (1/6,5/6) -- (0,1) -- cycle;
        \fill[opacity=.3,gray] (1,1/3) -- (5/6,1/6) -- (1,0) -- cycle;
        \draw[thick] (0.0,0.0) -- (1.0,0.5);
        \draw[thick] (0.0,0.5) -- (1.0,1.0);
        \draw[thick,dashed] (0.0,0.0) -- (0.5,1.0);
        \draw[thick,dashed] (0.5,0.0) -- (1.0,1.0);
      \end{axis}
    \end{tikzpicture}
    \caption{The varieties $\V{x_1 - x_2^2}$ (solid) and $\V{x_1^2 - x_2}$ (dashed)
      on the torus $\T^2$ parametrized by $\linterval{0,1}^2$.
      The shaded region designates where
      the polynomial $g$ from \ref{rem:lagrangecondition} is negative.}
    \label{fig:trigonometriclines}
  \end{figure}

  Observe that $\sform[σ]{x^α, 1} = σ\lr{x^{-α}} = 0$
  holds for all $α∈ℕ^2$.
  This implies that $\sform[σ]{q,1} = 0$ for all $q∈R$.
  Hence, for every choice of $d,d'$,
  the polynomial $p\coloneqq 1$ is contained in the kernel of the moment matrix
  $H_{d',d} \coloneqq
  \lr{\sform[σ]{w,v}}_{w ∈ B_{d'}^R,\,v ∈ B_d^R}$,
  where $B_d^R$ denotes a basis of $R_{≤d}$,
  with respect to any filtration of $R$.
  As $p = 1$ does not vanish on any non-empty variety,
  this shows that the statement of \ref{thm:idealequalkernelextended}
  does not hold for this matrix $H_{d',d}$
  with rows indexed by $B_d^R$ rather than $B_d^L$.

  Additionally, this shows that the kernel of the non-truncated map
  $R \to \semikdual{R}$, $p\mapsto \lr{q\mapsto \sform[σ]{q,p}}$,
  is not in general an ideal in $R$, in the trigonometric case.
  For instance, we have $\sform[σ]{x_2^2, x_1} = σ\lr{x^{\lr{1,-2}}} ≠ 0$,
  so $x_1 \notin \kernel H$, even though $1 ∈ \kernel H$.
\end{example}

\subsection{Non-negative measures}\label{sec:support:nonnegativemeasures}

In this \lcnamecref{sec:support:nonnegativemeasures},
we consider non-negative measures as well as
statements about signed measures that involve non-negative measures.
The non-negativity is an essential property
that allows us to state
the following \lcnamecref{thm:idealequalkernel}
which is a stronger version of \ref{thm:idealequalkernelextended}.
If $W = R_{≤d}$ is a component of the total degree filtration,
then in the affine case
this statement can also be obtained with a different proof by combining
\cite[Theorem~2.10]{laurentrostalski2012}
and \cite[Lemma~5]{lasserre2021:empiricalmomentschristoffel}.
\par
\begin{proposition}\label{thm:idealequalkernel}
  Let $μ$ be a non-negative measure on $Ω$ with finite moments,
  let $\mf{a} \coloneqq \Id{\supp μ} \subseteq L$
  be the vanishing ideal of (the Zariski closure of) its support
  and let $σ\colon L\to\K$ be its moment functional.
  Let $W\subseteq L$ be a $\K$-vector subspace.
  Then $\sform[σ]{\blank,\blank}$ induces a
  positive-definite form on $\submodulequotient{W}{\mf{a}}$.

  In particular, if $W$ is finite-dimensional and $B$ is a basis of $W$,
  let $H \coloneqq \lr{\sform[σ]{w,v}}_{w,v∈B}$.
  Then
  \[
    \mf{a} \cap W = \kernel H.
  \]
  Furthermore, $H$ is non-singular if and only if the elements of $B$
  are linearly independent modulo $\mf{a}\cap W$.
\end{proposition}
For the statement, only finiteness of the moments that occur in $H$ is needed,
so $σ$ must be defined on the subspace $\invol{W}\cdot W \subseteq L$.
\begin{proof}
  First observe that
  $\sform[σ]{\blank,\blank}$ is positive-semidefinite,
  as $\sform[σ]{p,p} = \int_{Ω} \abs{p(x)}^2 \d μ(x) ≥ 0$ for all $p∈L$.
  By \ref{lem:inducedformonquotient},
  $\sform[σ]{\blank,\blank}$ induces a form on $\submodulequotient{W}{\mf{a}}$
  and we need to show that it is non-degenerate.
  Assume that $p ∈ W$ is a polynomial such that $\sform[σ]{p,p} = 0$.
  Since $\abs{p}^2 ≥ 0$ on $Ω$, it follows from
  \cite[Proposition~1.23]{schmuedgen2017} that
  $\abs{p}^2$ vanishes on $\supp μ$
  and thus $p ∈ \mf{a}$.
  Hence, the induced form is non-degenerate
  and we have $\kernel H \subseteq \mf{a} \cap W$.

  Conversely, if $p ∈ \mf{a} \cap W$,
  then $(\invol{q} p)(ξ) = \invol{q}(ξ) p(ξ) = 0$ for all $ξ ∈ \supp μ$
  and all $q ∈ L$.
  Thus, in particular, we have
  $σ(\invol{q} p) = \int_{Ω} \invol{q}(x) p(x) \d μ(x) = 0$
  for all $q ∈ W$, so $p∈\kernel H$.

  From this, the addendum readily follows.
  If $H$ is non-singular, we have $\mf{a}\cap W = \kernel H = 0$,
  so the elements of $B$ are linearly independent modulo $\mf{a}\cap W$.
  If $H$ is singular, we find a non-trivial linear combination
  $q = \sum_{w∈B} q_w w ≠ 0$, $q_w∈\K$,
  with $q ∈ \kernel H = \mf{a}\cap W$,
  so $q \equiv 0 \pmod{\mf{a}\cap W}$.
\end{proof}

In particular, \ref{thm:idealequalkernel} holds with $W = R_{≤d}$ for any $d∈ℕ$,
so that
\[
  \mf{a} \cap R_{≤d} = \kernel H.
\]
Again, by Hilbert's basis theorem,
the ideal $\mf{a}$ is generated by $\mf{a} \cap R_{≤d}$
if $d$ is sufficiently large.
Hence, for such a number $d$,
the kernel of $H$ generates the ideal $\mf{a}$,
which is the statement of \cite[Theorem~2.10]{laurentrostalski2012},
so we can fully recover the ideal $\mf{a}$ from finitely many moments.

\begin{lemma}\label{lem:idealequalkernelnonneg}
  Let $\{F_d\}_{d∈ℕ}$ be a filtration of $R$ or $L$.
  Let $μ$ be a signed measure on $Ω$
  and $g ∈ F_δ$ for some $δ∈ℕ$
  such that $μ_+ = \invol{g} μ$
  is a non-negative measure with finite moments
  satisfying $\supp μ = \supp μ_+$.
  Then
  \[
    \Id{\supp μ} \cap F_d = \kernel H_{d+δ,d},
  \]
  for every $d∈ℕ$
  with $H_{d+δ,d} \coloneqq \lr{\sform[σ]{w,v}}_{w∈B_{d+δ},\,v∈B_d}$,
  where $σ\colon L\to\K$ denotes the moment functional of $μ$
  and $B_d,B_{d+δ}$ denote finite bases of $F_d,F_{d+δ}$, respectively.
\end{lemma}
\begin{proof}
  Observe that
  \begin{align}
    \Id{\supp μ} \cap F_d
    \subseteq \kernel H_{d+δ, d}
    &=
    \braced*{p ∈ F_d \;\middle|\; \int_{Ω} \invol{q} p \,\d μ = 0\text{\ for all $q ∈ F_{d+δ}$}}\\
    &\subseteq\label{eq:idealequalkernelpolynomial:inclusions}
    \braced*{p ∈ F_d \;\middle|\; \int_{Ω} \Invol{g q} p \,\d μ = 0\text{\ for all $q ∈ F_d$}},
  \end{align}
  where the last inclusion holds due to
  $g F_d \subseteq F_{d+δ}$.
  As $\invol{g} μ = μ_+$ is a non-negative measure on $Ω$,
  it follows from \ref{thm:idealequalkernel} that
  the set \myeqref{eq:idealequalkernelpolynomial:inclusions} is equal to
  $\Id{\supp μ_+} \cap F_d$.
  Then the statement follows from $\supp μ_+ = \supp μ$.
\end{proof}

For signed measures that are a product of a polynomial and a non-negative measure,
we then obtain the following result,
which in contrast to \ref{thm:idealequalkernelextended}
comes with an explicit bound on the size of the moment matrix
and does not require compactness of the support.
\par
\begin{corollary}\label{cor:idealequalkernelpolynomial}
  Let $\{F_d\}_{d∈ℕ}$ be a filtration of $R$ or $L$.
  Let $μ = g μ_+$ be a signed measure, where
  $μ_+$ denotes a non-negative measure on $Ω$ with finite moments
  and $g ∈ F_δ$ a polynomial for some $δ∈ℕ$.
  Then
  \[
    \Id{\supp μ} \cap F_d = \kernel H_{d+δ,d},
  \]
  for every $d∈ℕ$
  with $H_{d+δ,d} \coloneqq \lr{\sform[σ]{w,v}}_{w∈B_{d+δ},\,v∈B_d}$,
  where $σ\colon L\to\K$ denotes the moment functional of $μ$
  and $B_d,B_{d+δ}$ denote finite bases of $F_d,F_{d+δ}$, respectively.
\end{corollary}
\begin{proof}
  As $\invol{g} g$ and $g$ have the same vanishing set on $Ω$,
  it follows from \ref{lem:samevanishingonsupport}
  that $\supp\lr{\invol{g} g μ_+} = \supp\lr{g μ_+} = \supp μ$.
  As $\invol{g} g μ_+ = \invol{g} μ$ is a non-negative measure on $Ω$,
  the result follows from \ref{lem:idealequalkernelnonneg}.
\end{proof}

\begin{remark}\label{rem:pronyposdim}
  Under the assumptions that $μ_+$ is the uniform measure on some unknown variety $V\subseteq Ω$,
  that $g$ is non-zero on a Zariski-dense subset of $V$
  and that sufficiently large integers $d,δ∈ℕ$ are known
  such that $g∈F_δ$ and $V$ is generated by polynomials in $F_d$,
  then \ref{cor:idealequalkernelpolynomial} gives rise to
  a scheme for recovering all the defining data of $μ$ from finitely many of its moments.
  In particular, this includes finitely-supported measures as a special case,
  for which the variety $V$ is zero-dimensional.
  Hence, this may be regarded as an extension of Prony's method
  to more general measures.

  In this setting, we have $V = \supp μ = \supp μ_+$.
  Thus, we obtain the variety from
  $V = \V{\mf{a} \cap F_d} = \V{\kernel H_{d+δ,d}}$,
  where $\mf{a}\coloneqq\Id{\supp μ}\subseteq L$ denotes the vanishing ideal.
  Knowing $V$, one can compute the moments of the uniform measure $μ_+$ on $V$.
  Finally, finding $g$ is a linear problem involving only the moments of $μ$ and $μ_+$.
  Indeed,
  if $B_δ \subseteq F_δ$ represents a basis of $\submodulequotient{F_δ}{\mf{a}}$
  and $H \coloneqq \lr{\int_Ω \invol{w} v \,\d μ_+}_{w,v∈B_δ}$ is the corresponding moment matrix,
  we have
  \[
    H \resid{g} = \lr{\int_Ω \invol{w} \resid{g} \,\d μ_+}_{w∈B_δ} = \lr{\int_Ω \invol{w} \d μ}_{w∈B_δ},
  \]
  where $\resid{g} = \sum_{v∈B_δ} g_v v$ is the reduction of $g$ modulo $\mf{a}\cap F_δ$.
  As $H$ is a positive-definite matrix by \ref{thm:idealequalkernel},
  this linear system has a unique solution,
  so the polynomial $g$ is unique modulo $\mf{a}\cap F_δ$.

  Though, we remark that computing the moments of the uniform measure $μ_+$
  can be a difficult problem itself
  if the variety $V$ is not zero-dimensional.
  An approach that proved successful for us is
  to find a parametrization of the variety $V$
  and then compute the moments numerically with respect to this parametrization.
\end{remark}

We give a few examples of signed measures that illustrate
that the assumption of non-negativity is crucial
for \ref{thm:idealequalkernel}.
\par
\begin{example}\label{eg:negdensity}
  Let $μ$ be a signed measure supported on the real interval $[-1,1]\subseteq ℝ$ with density $g(x)\coloneqq x$
  and denote its moment functional by $σ\colon R \coloneqq ℝ[x] \to ℝ$,
  so that $σ(p) = \int_{-1}^1 p(x) g(x) \d x$ for $p∈R$.
  In particular, this means that $\sform[σ]{\blank,\blank}$ is not positive-semidefinite.
  One checks that, due to symmetry, the even moments $σ(x^{2α}) = 0$ vanish for $α∈ℕ$
  and thus $\det\lr{σ\lr{x^{2α+2β}}}_{0≤α,β≤d} = 0$ for all $d∈ℕ$.
  Then it follows that $\det\lr{σ\lr{x^{α+β}}}_{0≤α,β≤d} = 0$ if $d$ is even,
  for example using the Leibniz formula
  or by a suitable permutation of rows and columns.

  This means that, for every even $d$,
  we find some non-zero polynomial in $R_{≤d}$
  that lies in the kernel of the square moment matrix $\lr{σ\lr{x^{α+β}}}_{0≤α,β≤d}$,
  even though the variety corresponding to the Zariski closure of the support of the signed measure $μ$
  is the entire line $ℝ$, which is defined by the zero-ideal in $R$,
  and despite the fact that the monomials are linearly independent modulo the zero-ideal.
  Hence, the statement of \ref{thm:idealequalkernel} cannot hold.
  However, note that, in this example, the non-truncated Hankel operator
  is injective nevertheless,
  as stated in \ref{cor:idealequalkernelnontruncated}.
  Moreover, as $g$ is a polynomial of degree $1$,
  it follows from \ref{cor:idealequalkernelpolynomial}
  that the kernel of the rectangular matrix
  $\lr{σ\lr{x^{α+β}}}_{0≤α≤d+1,0≤β≤d}$ is zero, for every $d∈ℕ$.
\end{example}

In the affine setting with $L_{≤d} = R_{≤d}$, $d∈ℕ$,
and for a finitely-supported signed measure,
it follows from \ref{lem:vandermondefactorization}
that the statement of
\ref{thm:idealequalkernelextended} holds with $d'\coloneqq d$,
as long as $d∈ℕ$ is sufficiently large.
The following example shows that this can fail for small $d$.
\par
\begin{example}\label{ex:pointskernelunequaltruncatedideal}
  Let $R = \K[x]$ be the univariate polynomial ring
  and let $\mf{a} = \maxideal{ξ_1}\cap\maxideal{ξ_2}$
  with two distinct points $ξ_1,ξ_2∈\K$.
  We consider the map $σ = \ev_{ξ_1} - \ev_{ξ_2}$.
  Denote by $H_{d',d}$ the corresponding Hankel matrix, for $d,d'∈ℕ$.
  By \myeqref{eq:idealsubsetkernel},
  we have $\mf{a}\cap R_{≤d} \subseteq \kernel H_{d',d}$,
  but equality does not hold for small $d$.

  For instance, if $d'=d=0$,
  we have
  \[
    \mf{a}\cap R_{≤d} = 0 \subsetneq \kernel H_{0,0} = \kernel \lr{0}.
  \]
  However, if $d$ is sufficiently large, namely $d ≥ 2$, and if $d'≥d$, we have
  $\mf{a}\cap R_{≤d} = \kernel H_{d',d}$
  by \ref{lem:vandermondefactorization:2},
  regardless of the choice of $d'$.
\end{example}
\par
In contrast, we have seen in \ref{ex:signedmeasurewrongkernel}
that a similar statement is not possible
for infinitely-supported signed measures.
More precisely, it is an example in which one has
$\mf{a}\cap R_{≤d} ≠ \kernel H_{d,d}$ for all $d∈ℕ$,
since $1∈\kernel H_{d,d}$, but $1 \notin \mf{a}$.
For a non-negative measure, this would not be possible due to \ref{thm:idealequalkernel}.

For signed measures that are a complex linear combination of non-negative measures,
we obtain the following statement,
which in contrast to \ref{thm:idealequalkernelextended}
bounds the size of the moment matrix
and does not require compactness of the support.
\par
\begin{theorem}\label{thm:idealequalkernel:mixture}
  Let $μ = \sum_{j=1}^r λ_j μ_j$, where $λ_j∈ℂ\withoutzero$ and
  $μ_j$ are non-negative measures on $Ω$ with finite moments.
  Assume that $δ∈ℕ$ such that there exist elements
  $h_j ∈ L_{≤δ}$, $1≤j≤r$, such that
  $h_j ≥ 0$ on $Ω$
  and
  \[\label{eq:lagrangelikecondition}
    \supp\lr{h_j μ_k} = \begin{cases}
      \supp μ_k & \text{if $k = j$},\\
      ∅ & \text{otherwise}.
    \end{cases}
  \]
  Then
  \[
    \Id{\supp μ} \cap L_{≤d} = \kernel H_{d+δ,d},
  \]
  holds for all $d∈ℕ$
  with $H_{d+δ,d} \coloneqq \lr{\sform[σ]{w,v}}_{w∈B_{d+δ}, v∈B_d}$,
  where $σ\colon L\to\K$ denotes the moment functional of $μ$
  and $B_d,B_{d+δ}$ denote bases of $L_{≤d},L_{≤d+δ}$, respectively.
\end{theorem}
\begin{proof}
  Since $h_j μ_k = 0$ for all $k≠j$,
  we have
  \[\label{eq:unweightedmeasure}
    h_j μ = h_j λ_j μ_j = h_j λ_j μ_+,
  \]
  where we define $μ_+ \coloneqq \sum_{k=1}^r μ_k$.
  Letting $g \coloneqq \sum_{j=1}^r λ_j h_j ∈ L_{≤δ}$,
  we thus have
  $\invol{g} μ = \sum_{j=1}^r \abs{λ_j}^2 h_j μ_+$,
  which is a non-negative measure.
  Its support satisfies
  \[
    \supp\lr{\sum_{j=1}^r \abs{λ_j}^2 h_j μ_+}
    = \bigcup_{j=1}^r \supp\lr{h_j μ}
    = \supp μ,
  \]
  where the first equality holds due to \myeqref{eq:unweightedmeasure}
  and the second due to $\supp μ_j = \supp\lr{h_j μ_j}$.
  Hence, the statement follows from \ref{lem:idealequalkernelnonneg}.
\end{proof}

\begin{remark}\label{rem:lagrangecondition}
  Note that elements $h_j∈L_{≤δ}$ satisfying \myeqref{eq:lagrangelikecondition} exist,
  as long as $δ∈ℕ$ is large enough and
  the Zariski closures of $\supp μ_j$, $1≤j≤r$,
  are varieties such that each pair of them does not share a common irreducible component.
  This allows for elements $f_j∈L$
  such that $f_j$ vanishes on $\supp μ_k$ for all $1≤k≤r$ with $k≠j$
  and $f_j$ is non-zero on a dense subset of $\supp μ_j$,
  so we can choose $h_j \coloneqq \invol{f_j} f_j$, for $1≤j≤r$.

  In particular, we can apply \ref{thm:idealequalkernel:mixture}
  to \ref{ex:signedmeasurewrongkernel}
  with $f_1 \coloneqq x_1^2 - x_2$, $f_2 \coloneqq x_1 - x_2^2$.
  With $δ\coloneqq 2$, we then have $h_1,h_2 ∈ L_{≤δ}$
  in terms of the max-degree filtration
  and the hypotheses of the \lcnamecref{thm:idealequalkernel:mixture} are satisfied.
  The Laurent polynomial $g = h_1 - h_2$ constructed in the proof
  of the \lcnamecref{thm:idealequalkernel:mixture}
  is non-negative on one of the components
  and non-positive on the other, as depicted in \ref{fig:trigonometriclines},
  so that $\invol{g} μ$ is a non-negative measure.
\end{remark}

\minisec{Acknowledgments}

The author thanks Stefan Kunis for helpful comments and for discussions and support.
Parts of this manuscript are incorporated in the thesis \cite{wageringel2021}.

\setlength{\emergencystretch}{1em}%
\setcounter{biburlnumpenalty}{5000}  %
\printbibliography{}
\end{document}